\documentclass[psamsfonts]{amsart}

\usepackage{amssymb}
\usepackage{amsmath}
\usepackage{amsthm}
\usepackage[urlcolor=black, colorlinks=true, citecolor=black, linkcolor=black]{hyperref}
\usepackage[english]{babel}
\usepackage{graphics}

\usepackage{dsfont} 

\usepackage{mathtools}

\newcommand{ \R } { \mathbb{R} }
\newcommand{ \N } { \mathbb{N} }

\newcommand{\w}{\omega}
\newcommand{\wstar}{\omega^*}
\newcommand{\Cstar}{C^*}
\newcommand{\cont}{\mathfrak{c}}

\newcommand{\script}{\mathcal}
\newcommand{\parentheses}[1]{{\left( {#1} \right)}}
\newcommand{\p}{\parentheses}
\newcommand{\of}{\parentheses}
\newcommand{\tuple}{\parentheses}
\newcommand{\closure}[1]{\overline{#1}}

\newcommand{\interior}[1]{\mathrm{int}\of{#1}}

\newcommand{\Set}[1]{{\left\lbrace {#1} \right\rbrace}}
\newcommand{\singleton}{\Set}

\newcommand{\Union}{\bigcup}

\newcommand{\cardinality}[1]{{\left\lvert {#1} \right\rvert}}

\newcommand{\pair}[1]{\langle {#1} \rangle}
\def\set#1:#2{\Set{{#1} \colon {#2}}}

\newcommand{\continuum}{\mathfrak{c}}

\newcommand{\Cwstar}{C_k\p{\wstar,\wstar}}

\newcommand{\CkX}{C_k\p{X,X}}

\newcommand{\CompactOpen}[2]{\left [ #1, #2 \right ]}
\newcommand{\preimage}[2]{#2 ^{-1}\of{#1}}	
\newcommand{\image}[2]{#2\of{#1}}		
\newcommand{\collection}[1]{\script{#1}}	
\newcommand{\intersection}{\cap}
\newcommand{\Intersection}{\bigcap}
\newcommand{\explicitSet}[1]{\left \{ #1 \right \}}
\newcommand{\sequence}[2]{\left \langle #1 \; \colon \; #2 \right \rangle}
\newcommand{\remainder}[1]{#1^*}
\newcommand{\accumulation}[1]{{\closure{#1} \setminus {#1}}}
\newcommand{\convergesTo}{\rightarrow}
\newcommand{\eval}{\textrm{ev}}
\newcommand{\function}[3]{{#1} \colon {#2} \to {#3}}
\newcommand{\card}[1]{\left | #1 \right |}
\newcommand{\weight}[1]{w\of{#1}}
\newcommand{\restrict}[2]{#1\restriction{#2}}
\newcommand{\SC}[1]{\beta #1}

\newcommand{\id}{\operatorname{id}}

\newcommand{\SCinjectives}{S_1(\omega)}
\newcommand{\SCfins}{S(\omega)}
\newcommand{\mesh}{\#}
\newcommand{\cat}{^\frown}
\newcommand{\successor}[1]{\textnormal{succ}\p{#1}}
\newcommand{\operator}[1]{\textrm{#1}}
\newcommand{\conts}[2]{C\of{{#1},{#2}}}
\newcommand{\contsWstar}{\conts{\wstar}{\wstar}}

\begin{document}
\title{Self-maps under the compact-open topology}
\author{Richard J.\ Lupton, Max F.\ Pitz}
\address{Mathematical Institute\\University of Oxford\\Oxford OX2 6GG\\United Kingdom}
\email{richard.lupton@gmail.com}
\email{max.f.pitz@gmail.com}
\subjclass[2010]{Primary 54C35; Secondary 54D35, 54G05, 54E52}
\keywords{Compact-open topology, self-maps, Stone-\v{C}ech remainder, $\wstar$}

\begin{abstract}
This paper investigates the space $\Cwstar$, the space of continuous self-maps on the Stone-\v{C}ech remainder of the integers, $\wstar$, equipped with the compact-open topology. Our main results are that 
\begin{itemize}
\item $\Cwstar$ is Baire,
\item Stone-\v{C}ech extensions of injective maps on $\w$ form a dense set of weak $P$-points in $\Cwstar$, 
\item it is independent of ZFC whether $\Cwstar$ contains $P$-points, and that
\item $\Cwstar$ is not an $F$-space, but contains, as $\wstar$, no non-trivial convergent sequences. 
\end{itemize}
\end{abstract}

\maketitle
\thispagestyle{plain}

\newtheorem*{fundamentallemma}{Lemma~\ref{fundamental}}
\newtheorem{mythm}{Theorem} \numberwithin{mythm}{section} 
\newtheorem{myprop}[mythm]{Proposition}
\newtheorem{myobs}[mythm]{Observation}
\newtheorem{mycor}[mythm]{Corollary}
\newtheorem{mylem}[mythm]{Lemma} 
\newtheorem{myquest}[mythm]{Question} 
\newtheorem*{myconj}{Conjecture}
\newtheorem{mydef}[mythm]{Definition}
\newtheorem{myclaim}{Claim}
\newtheorem{myremark}[mythm]{Remark}
\newtheorem{mycase}{Case}
\newtheorem{mycase2}{Case}
\newtheorem{myexample}[mythm]{Example}

\section{Introduction}

Spaces of continuous functions are amongst the most natural and important objects in topology. This paper studies continuous \emph{self-maps} of topological spaces, and in particular continuous self-maps on $\wstar$, the Stone-\v{C}ech remainder of the integers.

The space $\wstar$ is one of the most important spaces in topology and its structure has been extensively examined. More importantly, however, our choice of $\wstar$, and our decision to study self-maps---as opposed to the more widely studied real-valued functions prevalent in topology and functional analysis---is motivated by the observation that $\contsWstar$ contains the Stone-\v{C}ech extensions of finite-to-one maps on $\w$. More precisely, the Stone-\v{C}ech extension of any finite-to-one map $\omega \to \omega$ restricts to a continuous map $\wstar \to \wstar$.

The finite-to-one maps $\w \to \w$ appear in important places in set-theoretic topology. M.E.\ Rudin proved in \cite{RudinComposants} for instance, that if there are two points $x,y \in \wstar$ such that for each finite-to-one map $\function{\phi}{\omega}{\omega}$, $\beta \phi (x) \neq \beta \phi (y)$, then $\mathbb{H}^*$ has at least two composants ($\mathbb{H}^*$ is the Stone-\v{C}ech remainder of the half-line $[0, \infty)$, and is a connected compact Hausdorff space (a continuum); the composant of a point $x$ is the union of all proper subcontinua containing that point). More generally, the \emph{Rudin-Blass order} on $\wstar$ is defined in terms of finite-to-one maps \cite{laflamme} (it is the finite-to-one version of the Rudin-Keisler order).

By equipping $\contsWstar$ with a topology, we can study the finite-to-one maps $\omega \to \omega$ (mod finite) as a topological space in its own right, which we denote $S(\w)$.  The compact-open topology is a natural topology to use here. Thinking of $\wstar$ as a space of ultrafilters on its Boolean Algebra, subsets of $\w$ modulo finite differences, an open set in the compact-open topology specifies where elements of the Boolean Algebra can be mapped. Since finite-to-one maps $\w \to \w$ are usually studied in relation to their action on ultrafilters, seeking a topology that interacts well with ultrafilters is sensible.

Besides studying extensions of finite-to-one maps, we also assess the extent to which properties of $\wstar$ are mirrored in $\contsWstar$, with respect to a suitable function space topology. For this, the compact-open topology again seems to naturally present itself. For example, in the context of self-maps on a locally compact Hausdorff space $X$, the compact-open topology is the smallest topology on $C(X,X)$ giving a topological semi-group such that the canonical embedding $X\hookrightarrow C(X,X)$, sending a point to the corresponding constant function, is an embedding \cite[VIII.1.9]{Magill}.

Once $\conts{X}{X}$ has been equipped with the compact-open topology, we denote it, in the standard way, by $\CkX$. In this paper we show that $\Cwstar$ and its subspace $S(\w)$ are both Baire spaces (Theorems~\ref{thm:CkBaire} and \ref{thm:FT1BaireSpace}). Further, we show that $S(\w)$ is a dense subspace of $\Cwstar$, all points of which are weak $P$-points in $\Cwstar$ (Theorems~\ref{thm:FT1dense} and \ref{denseweakPpoints}). More generally, we show in Theorem~\ref{autoweakPpoints} that every open finite-to-one map $X \to X$ on a compact Hausdorff, nowhere c.c.c.\ $F$-space $X$ is a weak $P$-point in $\CkX$.

Further, we show that for a zero-dimensional compact Hausdorff space $X$, a map in $\CkX$ is a $P$-point if and only if it has finite range, all points of which are $P$-points in $X$ (Theorem~\ref{consindepententPpoints}). Hence $\Cwstar$ has $P$-points precisely when $\wstar$ does, an assertion well-known to be independent of ZFC.

Lastly, we show that $\Cwstar$ is not an $F$-space (Theorem~\ref{thm:CkNotFspace}), but still contains, as does $\wstar$, no convergent sequences. Indeed, we prove that $\CkX$ never contains non-trivial convergent sequences, for any compact $F$-space $X$ (Theorem~\ref{noconvergentsequences}).

We would like to thank Rolf Suabedissen, Alan Dow and Jan van Mill for interesting discussions on the subject.

\section{$F$-spaces and $\wstar$}

This section contains a brief introduction to the spaces $\beta \w$ and $\wstar$. Recall that a subspace $Y \subset X$ is called $\Cstar$-embedded \index{C@$\Cstar$-embedded subspace} if every continuous real-valued bounded function on $Y$ can be extended to a continuous function on $X$. For every non-compact Tychonoff space $X$, its Stone-\v{C}ech compactification $\beta X$ is a compact Hausdorff space in which $X$ is dense and $C^*$-embedded, and $X^* = \beta X \setminus X$ is its remainder. For a concrete description of $\beta \w$ and $\wstar$ in terms of ultrafilters on the natural numbers see for example \cite{Rudin}. 

The space $\wstar$ is a zero-dimensional compact Hausdorff space of cardinality $2^\cont$ and weight $\cont=2^{\aleph_0}$, containing no isolated points. Every non-empty clopen subset of $\wstar$ is  again homeomorphic to $\wstar$. The space $\wstar$ contains a family of $\cont$ disjoint clopen sets and therefore has density $\cont$ \cite[3.6.18]{Eng}. 

Further, the space $\wstar$ has two additional crucial topological properties, the $G_\delta$- and the $F$-space property. Recall that a subset of a Tychonoff space of the form $f^{-1}(0)$ for some real-valued continuous function $f$ is called a \emph{zero-set}. A \emph{cozero-set} is the complement of a zero-set. \index{cozero-set} A space $X$ is called an $F$\emph{-space} if each cozero-set is $\Cstar$-embedded in $X$. A space is said to have the $G_\delta$-\emph{property} if every non-empty $G_\delta$-set has non-empty interior. 

Indeed, these properties are fairly common amongst Stone-\v{C}ech remainders. It is well-known that whenever $X$ is a locally compact $\sigma$-compact space then $X^*$ is an $F$-space \cite[14.16]{Ultrafilters}, and if $X$ is zero-dimensional, locally compact and $\sigma$-compact, then $X^*$ is compact zero-dimensional without isolated points and has the $G_\delta$-property \cite[14.17]{Ultrafilters}. 

A zero-dimensional compact space without isolated points with the $G_\delta$- and the $F$-space property is often called a Parovi\v{c}enko space. The reason why these properties have received special attention lies in the well-known result that under the Continuum Hypothesis, all Parovi\v{c}enko spaces of weight $\cont$ are homeomorphic to $\wstar$ \cite{parov}.

In the following, we list some more background results on $F$-spaces. Recall that subspaces $A,B \subset X$ are \emph{completely separated} if there is a continuous $f \colon X \to [0,1]$ such that $A \subset f^{-1}(0)$ and $B \subset f^{-1}(1)$. Equivalently, two subspaces are completely separated if they are contained in disjoint zero-sets. Proofs of the following results are contained in \cite[14.25]{Rings} and in the exercises \cite[14N]{Rings} and \cite[3.6.G]{Eng}.

\begin{enumerate}\itemsep0em 
\item A Tychonoff space is an $F$-space if and only if disjoint cozero-subsets are completely separated.
\item In an $F$-space, disjoint open $F_\sigma$-subsets have disjoint closures, and in normal spaces both conditions are equivalent.
\item Closed subspaces of normal $F$-spaces are $F$-spaces.
\item Infinite closed subspaces of compact $F$-spaces contain a copy of $\beta \w$. Therefore, compact $F$-spaces do not contain convergent sequences.
\end{enumerate}

\section{\texorpdfstring{A nice basis for $\CkX$}{A nice basis}}
The compact-open topology on the space $C\of{X,Y}$ of continuous functions $X \to Y$ is the topology generated by a subbasis consisting of sets of the form
\[ \CompactOpen{C}{U} = \set{f \in C\of{X,Y}}:{ \image{C}{f} \subseteq U} \]
where $C$ is a compact subset of $X$ and $U$ is an open subset of $Y$. The resulting topological space is denoted by $C_k(X,Y)$. A good reference for the basic properties of the compact-open topology is \cite[\S3.4]{Eng}.

In this section we prove that if $X$ is locally compact and zero-dimensional, the space $\CkX$ has a particularly nice basis, consisting of elements of the form $\bigcap_{i=0}^n [A_i,B_i]$ with $A_i,B_i \subset X$ compact clopen such that the $A_i$ are pairwise disjoint. 

\begin{mylem}
\label{subbasiclemma}
In a locally compact zero-dimensional space $X$, the collection of sets of the form $[A,B]$, for $A,B$ compact clopen subsets of $X$, forms a clopen subbasis for $\CkX$.
\end{mylem}

\begin{proof}
Let $\collection{B}$ be the collection of compact clopen subsets of $X$. Note that $\collection{B}$ is a base for $X$ which is closed under finite unions. It follows from \cite[XII.5.1]{Dugundji} that the collection $\set{[A,B]}:{A,B \in \script{B}}$ forms a clopen subbasis for $\CkX$.

To familiarise ourselves with the compact-open topology, we spell out the argument contained in the reference. Consider a subbasic open set $\CompactOpen{C}{U}$ in $\CkX$, with $C \subseteq X$ compact and $U \subseteq X$ open, and some $f \in \CompactOpen{C}{U}$. Since every set of the form $[A,B]$ for $A,B \in \collection{B}$ is open in $\CkX$, it is enough to show that there are $A, B \in \collection{B}$ with $f \in \CompactOpen{A}{B} \subseteq \CompactOpen{C}{U}$.

First, observe that $\image{C}{f}$ is a compact subset of $U$, so using that $\collection{B}$ is a base for $X$ closed under finite unions, we find $B \in \collection{B}$ with $\image{C}{f} \subseteq B \subseteq U$. Then $f \in \CompactOpen{C}{B} \subseteq \CompactOpen{C}{U}$. Now, $\preimage{B}{f}$ is open, and contains the compact set $C$, so using that $\collection{B}$ is a base for $X$ closed under finite unions again, we may find $A \in \collection{B}$ with $C \subseteq A \subseteq \preimage{B}{f}$. Furthermore $\image{A}{f} \subseteq B$, so $f \in \CompactOpen{A}{B}$. But clearly then
\[ f \in \CompactOpen{A}{B} \subseteq \CompactOpen{C}{B} \subseteq \CompactOpen{C}{U}\]
which is precisely what we required.

To see that $\CompactOpen{A}{B}$ is clopen, note that for $f \notin [A,B]$ there is $x \in A$ such that $f(x) \notin B$. Thus, $[\singleton{x},X\setminus B]$ is a neighbourhood of $f$ not intersecting $[A,B]$.
\end{proof}

\begin{mylem}
\label{prettycoollemma}
Suppose $A_i, B_i$ are clopen in $X$ for $i \in \explicitSet{0,\dots,n}$. Then there are clopen $U_i, V_i$ for $i \in \explicitSet{0,\dots, m}$ so that in $C_k\of{X, X}$,
\[ \Intersection_{i=0}^n \CompactOpen{A_i}{B_i} = \Intersection_{i=0}^m \CompactOpen{U_i}{V_i}, \]
and such that $\explicitSet{U_0, \dots, U_m}$ is a pairwise disjoint refinement of $\explicitSet{A_0, \dots, A_n}$.
\end{mylem}

\begin{proof}
We work by induction on $n$. For $n=0$, there is nothing to prove. For $n=1$, observe that \[ \CompactOpen{A_0}{B_0} \intersection \CompactOpen{A_1}{B_1} = \CompactOpen{A_0 \intersection A_1}{B_0 \intersection B_1} \intersection \CompactOpen{A_0 \setminus A_1}{B_0} \intersection \CompactOpen{A_1 \setminus A_0}{B_1}. \]

So let $n \geq 2$, and suppose the Lemma holds at $n-1$. Consider the basic open set $\Intersection_{i=0}^n \CompactOpen{A_i}{B_i}$. Applying the inductive hypothesis to $\Intersection_{i=0}^{n-1} \CompactOpen{A_i}{B_i}$, we may assume without loss of generality that the collection $\explicitSet{A_0, \dots, A_{n-1}}$ is already pairwise disjoint. Generalising our observation from the case $n=2$, we obtain
\[ \CompactOpen{A_n}{B_n} \intersection \Intersection_{i=0}^{n-1} \CompactOpen{A_i}{B_i} = \CompactOpen{A_n\setminus \Union_{i=0}^{n-1} A_i}{B_n} \intersection \Intersection_{i=0}^{n-1} \CompactOpen{A_n\intersection A_i}{B_n \intersection B_i} \intersection \Intersection_{i=0}^{n-1} \CompactOpen{A_i \setminus A_n}{B_i}, \]
from which the result follows.
\end{proof}

\begin{mythm}
\label{bigintersectionresult}
Suppose $X$ is locally compact and zero-dimensional. Then $C_k\of{X, X}$ has a base of sets of the shape
\[ \Intersection_{i=0}^n \CompactOpen{A_i}{B_i}, \]
for $A_i, B_i$ compact clopen in $X$, and $\explicitSet{A_0, \dots, A_n}$ pairwise disjoint.
\end{mythm}

\begin{proof}
Clear by Lemmas~\ref{subbasiclemma} and~\ref{prettycoollemma}.
\end{proof}

By similar considerations one can prove that for a locally compact, zero-dimensional space $X$, the space $C_k\of{X, X}$ has a $\pi$-base consisting of sets of the shape
$ \Intersection_{i=0}^n \CompactOpen{C_i}{D_i}$
where $\explicitSet{C_0, \dots, C_n}$ and $\explicitSet{D_0, \dots, D_n}$ are both pairwise disjoint collections of compact clopen subsets of $X$, \cite{LuptonThesis}.

\section{First topological properties of spaces of self-maps} 
The above results about bases and $\pi$-bases of $\CkX$ allow us to make first observations about topological properties of $\Cwstar$, and more generally $\CkX$. 

Recall that for all Tychonoff spaces $X$ the function space $\CkX$ is also Tychonoff \cite[3.4.15]{Eng}.

\begin{myobs}
For every locally compact zero-dimensional space $X$, the weight of  $\CkX$ equals the weight of $X$.
\end{myobs}

\begin{proof}
By Lemma~\ref{subbasiclemma}.
\end{proof}

\begin{myobs}
For a non-empty space $X$, the cellularity of $\CkX$ is at least as big as the cellularity of $X$. 
\end{myobs}

\begin{proof}
Let $x \in X$, and $\set{A_\alpha}:{\alpha < \kappa}$ be a collection of disjoint open subsets of $X$. Then the family $\set{[\singleton{x},A_\alpha]}:{\alpha < \kappa}$ is a $\kappa$-sized collection of disjoint open subsets of $\CkX$.
\end{proof}

\begin{myobs}
$C_k\of{\remainder{\omega}, \remainder{\omega}}$ has density $\continuum$.
\end{myobs}

\begin{proof}
Because $\weight{C_k\of{\remainder{\omega}, \remainder{\omega}}} = \continuum$, $C_k\of{\remainder{\omega}, \remainder{\omega}}$ has density at most $\continuum$. By the previous result, $C_k\of{\remainder{\omega}, \remainder{\omega}}$ has density at least $\continuum$.
\end{proof}

\begin{myobs}
$\card{C\of{\remainder{\omega}, \remainder{\omega}}} = 2^\continuum$.
\end{myobs}

\begin{proof}
There are $2^\continuum$ constant functions. Also, $\remainder{\omega}$ is Hausdorff so continuous functions are determined completely by their action on a dense subset of $\remainder{\omega}$. Given $\remainder{\omega}$ has density $\continuum$ we have
\[ 2^\continuum \leq \card{C\of{\remainder{\omega}, \remainder{\omega}}} \leq \card{\p{\remainder{\omega}}^\continuum} = \p{2^\continuum}^\continuum = 2^\continuum. \qedhere \]
\end{proof}

\begin{myobs}
If $X$ is an infinite compact zero-dimensional Hausdorff space then $\CkX$ contains an infinite locally finite family of disjoint non-empty open sets.
\end{myobs}

\begin{proof}
Let $A \subset X$ be an infinite clopen set with non-empty complement $B=X \setminus A$, and fix a collection $\set{A_n}:{n \in \w}$ of disjoint non-empty clopen subsets of $A$. For $n \in \w$ define the non-empty (basic) open sets 
$$U_n = [A_n,B] \cap [A\setminus A_n,A].$$
We claim the collection $\script{U}=\set{U_n}:{n \in \w}$ is locally finite. Indeed, suppose that $f \notin \bigcup \script{U}$. 
If $f(A) \subseteq A$, then $[A,A]$ is a neighbourhood of $f$ witnessing that $[A,A] \cap U_n = \emptyset$ for all $n \in \w$. 

Otherwise, if $f(A) \not\subseteq A$, then $f^{-1}(B) \cap A \neq \emptyset$. Note that by continuity of $f$, the set $f^{-1}(B)$ is clopen. If for some $n \in \w$ we have $A'_n=f^{-1}(B) \cap A_n \neq \emptyset$ then $[A'_n,B]$ is a neighbourhood of $f$ witnessing that $[A',B] \cap U_m = \emptyset$ for all $m \in \w \setminus \singleton{n}$. And finally, if $f^{-1}(B) \cap A_n = \emptyset$ for all $n \in \w$ then $[f^{-1}(B) \cap A,B]$ is a neighbourhood of $f$ such that $[f^{-1}(B) \cap A,B] \cap U_n = \emptyset$ for all $n \in \w$. 
\end{proof}

In fact, if $X$ contains a family of disjoint open sets of size $\kappa$, an easy modification shows that under the above conditions, $\CkX$ contains a locally finite collection of disjoint open sets of size $\kappa$.

\begin{mythm}
\label{notcountcpt}
The function space $\CkX$ of an infinite compact zero-dimensional Hausdorff space $X$ is not pseudocompact.
\end{mythm}

\begin{proof}
By \cite[3.10.22]{Eng}, for a Tychonoff space $Y$, pseudocompactness is equivalent to the assertion that every locally finite family of non-empty open subsets of $Y$ is finite. Hence, $\CkX$ is not pseudocompact by the previous observation. 
\end{proof}

Since pseudocompactness is implied by (countable) compactness \cite[3.10.20]{Eng}, it follows that the function space $\CkX$ of an infinite compact zero-dimensional Hausdorff space $X$ is never (countably) compact.

\begin{mythm}
Every pseudocompact subspace of $\Cwstar$ has empty interior. 
\end{mythm}

\begin{proof}
Note that as a consequence of Theorem~\ref{bigintersectionresult} every open set contains a basic clopen subset homeomorphic to a finite product of $\Cwstar$, which is not pseudocompact by the previous corollary. Hence, the original set could not have been pseudocompact, as pseudocompactness is hereditary with respect to clopen subspaces.
\end{proof}

\section{\texorpdfstring{Completeness properties of $\Cwstar$}{Completeness properties of function spaces}}

Establishing completeness properties of function spaces $C_k(X,\R)$ of real-valued functions is a natural but hard problem. Indeed, no complete characterisation is known for which spaces $X$ the function space  $C_k(X,\R)$ is Baire. For a characterisation when $C_k(X,\R)$ is Baire for locally compact or first countable spaces $X$ see \cite{GM}. Completeness results have also been established for other target spaces: If $X$ is a hemi-compact $k$-space, and $Y$ \v{C}ech-complete with a $G_\delta$ diagonal then $C_k(X,Y)$ is \v{C}ech-complete \cite[4.1]{Hola}.

We have already seen that $\Cwstar$ does not exhibit many of the most common compactness properties. In this section however, we show that $\Cwstar$ is Baire. In fact, we establish something slightly stronger, namely that $\Cwstar$ is (strongly) Choquet. 

Recall that the property of being Baire can be described using the Choquet game. The Choquet game is an infinite ($\omega$)-length game with two players, called E and NE. The players take turns to choose non-empty open sets, with the condition that if one player chooses an open set $U$, then the other player on their subsequent turn must choose a non-empty open set $V$ with $V \subseteq U$. Player E begins. Plays of the game therefore form descending $\omega$-length chains of open sets, of the shape $\tuple{U_1, V_1, U_2, V_2, \dots}$, where the $U_i$ correspond to E's moves, and the $V_i$ correspond to NE's moves. Player E wins if the resulting intersection, $\Intersection_{i \in \omega} U_i = \Intersection_{i \in \omega}V_i$, is empty. Otherwise, player NE wins. By a theorem of Oxtoby, the space $X$ is Baire if and only if E has no winning strategy (\cite[I.8.11]{Kechris}). If NE has a winning strategy, then $X$ is said to be Choquet (\cite[I.8.12]{Kechris}). Clearly if $X$ is Choquet, then $X$ is also Baire.

It will be convenient for the second part of this section to formalise precisely what is meant by a strategy. The following approach is taken from \cite{Kechris}. Observe that partial plays of the Choquet game are finite descending sequences of non-empty open sets that can be given the structure of a tree, where we say $s \leq t$ precisely when $t$ extends $s$, that is $s \subseteq t$. A strategy $\sigma$ is then just a special kind of subtree (we demand a subtree be closed under taking initial segments, as in \cite{Kechris}). More precisely, a tree $\sigma$ is a strategy (for NE) if and only if
\begin{enumerate}
\item $\emptyset \in \sigma$;
\item if $\tuple{U_0, \dots, U_n, V_n} \in \sigma$ then for any non-empty open subset $U_{n+1}$ of $V_n$, $\tuple{U_0, \dots, U_n, V_n, U_{n+1}} \in \sigma$;
\item if $\tuple{U_0, \dots, U_n} \in \sigma$, then there is precisely one non-empty open set $V_n \subseteq U_n$ with $\tuple{U_0, \dots, U_n, V_n} \in \sigma$; we write $V_n = \sigma \of{U_0, \dots, U_n}$.
\end{enumerate}
Observe that the branches of a strategy $\sigma$ correspond to plays of the Choquet game where NE has adhered to the strategy $\sigma$. We call such plays $\sigma$-compatible. A strategy is therefore a winning strategy precisely when all its branches are winning plays.

This formalised notion of a strategy is useful when we need to be careful about how we construct our strategy, and we will use this description explicitly later. However, for the first results in this section, one can think of a strategy like a function, which takes the history of the game played so far, and provides the next move for the player. Clearly such a description can be formalised as above. In fact, we will simply describe how NE should respond to the history of the game, since clause (3) in the definition above is the only clause we might have control over.

In fact, when we describe NE's moves, it will often only depend on E's previous move. A strategy (for NE) which is not dependent on the entire history of the game, but only the previous move of E, is called a $1$-tactic. If NE has a winning $1$-tactic in the Choquet game on $X$, then clearly NE has a winning startegy, but the converse is not true in general (see, for example, \cite{Debs}).

The strong Choquet game is a variant of the Choquet game, where E may specify on each of their turns a point inside the open set that they play, and NE must then respond with an open subset containing this point. The winning condition is the same. Strategies can be defined analogously to the Choquet game, as can $1$-tactics. A space $X$ is strongly Choquet if and only if NE has a winning strategy in the strong Choquet game on $X$. Every strongly Choquet space is Choquet.

We will later show that the subspace of Stone-\v{C}ech extensions of finite-to-one maps is Choquet. Let us begin with a result showing that we can tailor injective maps $\w \to \w$ such that their Stone-\v{C}ech extensions satisfy countably many conditions imposed by the compact-open topology. To do so, we adapt the notion of Cantor schemes and Lusin schemes used in \cite{Kechris}. For a tree $\pair{T, \leq}$ of height $\omega$ write $T_n= \set{t \in T}:{\textrm{height}(t) = n}$ for $n \in \w$, and denote the set of successors of an element $t \in T_n$ by $\successor{t}=\set{s \in T_{n +1}}:{t \leq s}$. We call $T$ a \emph{finite splitting tree} if $0 < \cardinality{\successor{t}} < \infty$ for all $t \in T$. 

\begin{mydef}
Let $T$ be a tree of height $\w$. A collection $\set{A_t}:{t \in T}$ of non-empty clopen subsets of a space $X$ is called a $T$-\emph{scheme} (in $X$) if
\begin{enumerate}
\item $A_t \cap A_s = \emptyset$ for all $n \in \w$ and $t,s \in T_n$ with $t \neq s$, and 
\item $\bigcup_{s \in \successor{t}} A_s \subseteq A_t$ for all $t \in T$.
\end{enumerate}
If in addition, a $T$-scheme $\set{A_t}:{t \in T}$ also satisfies 
\begin{enumerate}
\setcounter{enumi}{2}
\item $\bigcup_{t \in T_n} A_t = X$ for all $n \in \w$,
\end{enumerate}
we refer to it as a \emph{covering} $T$-\emph{scheme}. Lastly, if the collection $\set{A_t}:{t \in T}$ only satisfies $(2)$, we refer to it as a \emph{weak} $T$-\emph{scheme}. 
\end{mydef}

Under this notation, a Cantor scheme is a $2^{<\w}$-scheme, and a Lusin scheme is a $\w^{<\w}$-scheme. 

For the next result, since when discussing Stone-\v{C}ech extensions of injective maps $\w \to \w$ we are primarily interested in the restriction of said maps to $\remainder{\omega}$, let us for notational convenience cease to distinguish between $\SC{\phi}$ and $\restrict{\SC{\phi}}{\remainder{\omega}}$.

\begin{mylem}
\label{lem:recursiveconstr}
Let $\pair{T, \leq}$ be a finite splitting tree of height $\omega$ and suppose that $\set{A_t}:{t \in T}$ is a covering $T$-scheme and $\set{B_t}:{t \in T}$ is a weak $T$-scheme in $\wstar$.  Then there is an injective map $\phi \colon \w \to \w$ such that its Stone-\v{C}ech extension $\beta \phi$ satisfies $\beta \phi(A_t)\subseteq B_t$ for all $t \in T$. 
\end{mylem}

\begin{proof}
Using \cite[3.6.A]{Eng}, fix collections $\set{C_t}:{t \in T}$ and $\set{D_t}:{t \in T}$ of (clopen) subsets of $\w$ such that $\set{C_t}:{t \in T}$ is a covering $T$-scheme in $\w$, $\set{D_t}:{t \in T}$ is a weak $T$-scheme in $\w$, and $C^*_t (= \closure{C}_t \setminus C_t) = A_t$ and $D^*_t = B_t$ for all $t \in T$.

We construct an injective function $\phi \colon \w \to \w$ such that, for each $m \in \omega$ and $t \in T_m$,
$$ (\star) \quad \phi\restriction_{n+1}\p{C_t \cap [m,n]} \subseteq D_t \textnormal{ for all } n \in \omega \textnormal{ with }m \leq n.$$
In other words, $\phi\restriction_n$ promises to send $C_t$ to $D_t$ whenever $\operator{height}\of{t} < n$. Since $\phi$ is injective it extends to a continuous self-map of $\wstar$ \cite[3.7.16]{Eng}, and satisfies $\beta \phi(A_t) \subseteq B_t$ for all $t \in T$, because by $(\star)$ the set $\phi(C_t)$ is almost contained in $D_t$, i.e.\ $\cardinality{\phi(C_t) \setminus D_t } < \infty$. 

Since $\set{C_t}:{t \in T}$ is a covering scheme, for every $n \in \w$ the set $\set{t \in T_n}:{n \in C_t}$ contains a unique element $t_n$. We define $\phi \colon \w \to \w$ recursively by
$$\phi\colon n \mapsto \min \p{D_{t_n} \setminus \textnormal{ran}\p{\phi\restriction_n}}.$$ 
Since $D^*_t = B_t \neq \emptyset$ for all $t \in T$, every $D_t$ is infinite and hence $\phi$ is well-defined, and injective. To see that condition $(\star)$ is satisfied, we proceed by induction. Suppose the statement holds for $\phi\restriction_{n}$. Let $m \leq n$ and consider some $C_t$ for $t \in T_m$. If $n \notin C_t$ then 
$$\phi\restriction_{n+1}\p{C_t \cap [m,n]} = \phi\restriction_{n+1}\p{C_t \cap [m,n-1]} = \phi\restriction_{n}\p{C_t \cap [m,n-1]} \subseteq D_t$$ by induction assumption. And if $n \in C_t$, we have $C_{t_n} \subseteq C_t$ by properties $(1)$ and $(2)$ of schemes, which in turn implies $D_{t_n} \subseteq D_t$.
Hence  
$$\phi\restriction_{n+1}\p{C_t \cap [m,n]} = \phi\restriction_{n}\p{C_t \cap [m,n-1]} \cup \phi(n) \subseteq D_t \cup D_{t_n} \subseteq D_t.$$
This completes the inductive step and proof.
\end{proof}

\begin{mycor}
\label{thm:FT1dense}
The set $\set{\SC{\phi}}:{\function{\phi}{\omega}{\omega} \textrm{ is injective}}$ is dense in $C_k\of{\remainder{\omega}, \remainder{\omega}}$. \qed
\end{mycor}

The density result was first proved by an alternative method in \cite{LuptonThesis}.

Let us now consider the basis $\script{B}$ for $\Cwstar$ as described in Theorem~\ref{bigintersectionresult}, consisting of sets of the form
$$U = \Intersection_{i=0}^n \CompactOpen{A_i}{B_i} $$
for $A_i,B_i$ clopen subsets of $\wstar$ for all $0 \leq i \leq n$ and $\explicitSet{A_0, \dots, A_n}$ a clopen partition of $\wstar$.

\begin{mylem}
\label{fairlyclearlemma}
Let $\script{B}$ be the base for $\Cwstar$ described above, and suppose $V= \Intersection_{i=0}^k\CompactOpen{C_i}{D_i} \in \script{B}$. Then every $U \in \script{B}$ with $U \subseteq V$ can be written as  
$U = \Intersection_{i=0}^j\CompactOpen{A_i}{B_i} $ such that $\explicitSet{A_0, \dots, A_k}$ forms a clopen partition of $\wstar$, and $\explicitSet{A_0, \dots, A_j}$ refines $\explicitSet{C_0, \dots, C_k}$, while $\explicitSet{B_0, \dots, B_j}$ refines $\explicitSet{D_0, \dots, D_k}$ in such a way that $A_l \subset C_m$ implies $B_l \subset D_m$.
\end{mylem}

\begin{proof}
By applying Lemma~\ref{prettycoollemma} to the set $U= U \cap V$ it follows that without loss of generality, the collection $\explicitSet{A_0, \dots, A_j}$ is a partition of $\wstar$ refining $\explicitSet{C_0, \dots, C_k}$.

To see that $\explicitSet{B_0, \dots, B_j}$ refines $\explicitSet{D_0, \dots, D_l}$, consider say $B_0$. Then $A_0 \subseteq C_m$ for some $m \leq k$ by the previous part. If $B_0 \not\subseteq D_m$ fix $y_0 \in B_0 \setminus D_m$ and $y_i \in B_i$ for $1 \leq i \leq j$ and consider $f$, the continuous finite-range function sending $A_i \mapsto y_i$. Then, it is clear that $f \in U \setminus V$, a contradiction.
\end{proof}

\begin{mythm}
\label{thm:intersections}
Let $\script{B}$ be the base for $\Cwstar$ described above, and suppose $\set{U_n}:{n \in \w} \subset \script{B}$ is a nested collection of basic open sets, i.e.\ $U_{n+1} \subset U_n$ for all $n \in \w$. Then $\bigcap_{n \in \w} U_n$ contains the Stone-\v{C}ech extension of an injection $\omega \to \omega$, and in particular is non-empty. 
\end{mythm}

\begin{proof}
Using Lemma~\ref{fairlyclearlemma}, we can write
$$U_n = \Intersection_{i=0}^{j_n}\CompactOpen{A^n_i}{B^n_i} $$
such that for the appropriate finite splitting tree $T$, the collections 
$$\set{A^n_i}:{n \in \w, i \leq j_n} \; \textnormal{ and } \; \set{B^n_i}:{n \in \w, i \leq j_n}$$ form a covering $T$-scheme and a weak $T$-scheme in $\wstar$ respectively. 

By Lemma~\ref{lem:recursiveconstr} there is an injective map $\phi \colon \w \to \w$ such that its Stone-\v{C}ech extension satisfies $\beta\phi (A^n_i) \subseteq B^n_i$ for all $n \in \w$ and $0 \leq i \leq j_n$. It follows that $\beta \phi \in U_n$ for all $n \in \w$ as desired.
\end{proof}

\begin{mycor}
$\Cwstar$ is strongly Choquet. Moreover, NE has a winning $1$-tactic for the strong Choquet game, and a winning $1$-tactic in the Choquet game, and in both cases can always obtain a Stone-\v{C}ech extension of an injection $\omega \to \omega$ in the winning set.
\end{mycor}

\begin{proof}
The winning 1-tactic for player NE looks as follows. Whenever player E plays an open set $V_n$ and a point $x \in V_n$, player NE responds with any basic open set $U_n \in \script{B}$ such that $x \in U_n \subseteq V_n$. It follows from Theorem~\ref{thm:intersections} that $\bigcap_{n \in \w} U_n \neq \emptyset$.
\end{proof}

\begin{mycor}	\label{thm:CkBaire}
$\Cwstar$ is Choquet and Baire. \qed
\end{mycor}

With a little more analysis we can say more about the properties of Stone-\v{C}ech extensions of injections $\omega \to \omega$, and finite-to-one maps $\omega \to \omega$, as subspaces of $\Cwstar$. Let us write $\SCinjectives$ for the set of Stone-\v{C}ech extensions of injective maps $\omega \to \omega$, and recall that we denote the set of Stone-\v{C}ech extensions of finite-to-one maps $\omega \to \omega$ by $\SCfins$.

\begin{mydef}
Suppose $X$ is a (Choquet) space and $T$ a subset of $X$. We say NE has a winning strategy that targets $T$ (or more succinctly, $T$ is targetable in $X$) if and only if NE has a winning strategy $\sigma$ such that whenever $\p{U_0, V_0, U_1, \dots}$ is a $\sigma$-compatible play of the Choquet game on $X$, then
\[ \Intersection_{n \in \omega}U_n \intersection T \neq \emptyset. \]
\end{mydef}

Rephrasing the above we see that in the Choquet game on $\Cwstar$, player NE has a winning strategy that targets $\SCinjectives$. We now wish to prove a general theorem about such targetable subsets. To do this we introduce some machinery for building winning strategies. This machinery was first introduced in \cite{LuptonThesis}. Recall that strategies can be formally defined as trees; this formalism will be used in the following. Note that a finite sequence of length $n$ will be viewed as a function on the set $n = \explicitSet{0, \dots, n-1}$.

\begin{mydef}
Suppose $\sigma$ and $\mu$ are strategies for NE (E) in the Choquet game on $X$. A transfer map (from $\mu$ to $\sigma$) is a map $\function{\script{T}}{\mu}{\sigma}$ such that
\begin{enumerate}
\item $\p{s \subseteq t \rightarrow \script{T}\of{s} \subseteq \script{T}\of{t}}$
\item $\script{T}$ preserves length, i.e. for all $s \in \mu$, $\textrm{length}\of{s}=\textrm{length}\of{\script{T}\of{s}}$.
\end{enumerate}

If $\script{T}$ is a transfer map from $\mu$ to $\sigma$, then for $s$ a $\mu$-compatible sequence of open sets, we may (abusing notation) define $\script{T}\of{s}$ to be $\Union_{n \in \omega}\script{T}\of{\restrict{s}{n}}$ (where $n$-tuples here are thought of as partial functions on the domain $\omega$; this just gives us the obvious limit).
\end{mydef}

Transfer maps are useful because of the following observations, which will be used to check that new strategies that we build are winning.

\begin{myobs}
Suppose $\script{T}$ is a transfer map from $\mu$ to $\sigma$. If $s$ is $\mu$-compatible, then $\script{T}\of{s}$ is $\sigma$-compatible. Further, if $s$ is $\mu$-compatible, then for all $n\in\omega$ we have $\script{T}\of{\restrict{s}{n}} = \restrict{\script{T}\of{s}}{n}\subseteq \script{T}\of{s}$.
\end{myobs}

In the following, $\cat$ will be used to denote the concatenation operator on finite sequences.

\begin{mythm}   \label{thm:targetProperties}
Suppose $X$ is regular, and $D$ is a dense subset of $X$. Suppose NE has a winning strategy which targets $D$ in the Choquet game on $X$. Then
\begin{enumerate}
\item $D$ is non-meager in $X$;
\item NE has a winning strategy in the Choquet game on $D$.
\end{enumerate}
\end{mythm}

\begin{proof}
(1) is easy to verify, so we focus on (2). First observe that the set $\script{R}$ of regular open subsets of $X$ forms a base for $X$. It is easily verified that NE has a winning strategy that targets $D$ in the Choquet game on $X$ where both players, E and NE, are restricted to playing elements of $\script{R}$. We also have that $\script{R} \mesh D = \set{U \intersection D}:{U \in \script{R}}$ is a base for the topology on $D$, and it suffices to show that NE has a winning strategy in the Choquet game on $D$ where moves are restricted to open sets from $\script{R} \mesh D$. For $U \in \script{R} \mesh D$, let us define $L\of{U} = \interior{\closure{U}^X}$, which is regular open in $X$, and observe that $L\of{U} \intersection D = U$, and if $U \subseteq V$ then $L\of{U} \subseteq L\of{V}$.

So let $\sigma$ be a winning strategy for NE targetting $D$ in the Choquet game on $X$ with moves restricted to $\script{R}$. We build a winning strategy $\mu$ for NE in the Choquet game on $D$ with moves from $\script{R} \mesh D$, alongside a transfer map $\mu \to \sigma$, denoted $\script{T}$. $\script{T}$ will have the following properties;
\begin{enumerate}
\item If $s = \p{U_0, V_0, \dots, U_n} \in \mu$ then $\script{T}\of{s} = t \cat \p{L\of{U_n}}$ for some $t$;
\item If $s = \p{U_0, V_0, \dots, V_n} \in \mu$ then $\script{T}\of{s} = t \cat \p{L\of{V_n}}$ for some $t$.
\end{enumerate}
Provided $\mu$ and $\script{T}$ can be defined in this way, then whenever $s = \p{U_0, V_0, \dots}$ is a $\mu$-compatible play of the Choquet game on $D$ with moves from $\script{R} \mesh D$, then $\script{T}\of{U_0, V_0, \dots} = \p{L\of{U_0}, L\of{V_0}, \dots}$, is $\sigma$-compatible, and so
\[ \Intersection_{n \in \omega} L\of{U_n} \intersection D \neq \emptyset. \]
But
\[ \Intersection_{n \in \omega} L\of{U_n} \intersection D = \Intersection_{n \in \omega} U_n \]
so $\mu$ is a winning strategy for NE in the Choquet game on $D$ with moves from $\script{R} \mesh D$. Hence $D$ is Choquet.

So we are left with the task of justifying the recursive construction of $\mu$ and $\script{T}$. We are required to have $\emptyset \in \mu$ and $\script{T} \of{\emptyset} = \emptyset$. Also for any $U_0 \in \script{R} \mesh D$, we have $\p{U_0} \in \mu$ and $\script{T}\of{U_0} = \p{L\of{U_0}}$. Now if $\p{U_0, V_0, \dots, V_n} \in \mu$, then for any $U_{n+1} \in \script{R} \mesh D$ with $U_{n+1} \subseteq V_n$ we have $\p{U_0, V_0, \dots, V_n, U_{n+1}} \in \mu$, and $\script{T}\of{U_0, V_0, \dots, V_n, U_{n+1}} = \script{T}\of{U_0, V_0, \dots, V_n} \cat \p{L\of{U_{n+1}}}$. Provided (1) and (2) have been staisfied in the recursion so far, this is well-defined. Now suppose $\p{U_0, V_0, \dots, U_n} \in \mu$. Then let $V_n^\prime = \sigma\of{\script{T}\of{U_0, V_0, \dots, U_n}}$, and set $V_n = V_n^\prime \intersection D$. Then $V_n \in \script{R} \mesh D$, $V_n \subseteq U_n$, and $L\of{V_n} = V_n^\prime$. We then insist that $\p{U_0, V_0, \dots, U_n, V_n} \in \mu$, and
\[ \script{T}\of{U_0, V_0, \dots, U_n, V_n} = \script{T}\of{U_0, V_0, \dots, U_n} \cat \p{L\of{V_n}}. \]
This completes the construction of $\mu$ and $\script{T}$.
\end{proof}

\begin{myremark}    \label{rem:targetableStrengthenings}
Observe in the above, that if $\sigma$ is a winning $1$-tactic, then so is $\mu$. Furthermore, one obtains a similar result where ``Choquet'' is replaced by ``strongly Choquet'', by adapting the above proof.
\end{myremark}

Conversely, if $D$ is a dense Choquet subspace of $X$, then NE has a winning strategy targeting $D$ in the Choquet game on $X$ (\cite[1.2.3]{LuptonThesis}). Hence we have

\begin{mythm}    \label{thm:classTargetables}
Let $X$ be a regular (Choquet) space, and $D$ a dense subspace of $X$. Then $D$ is Choquet if and only if NE has a winning strategy in the Choquet game on $X$ which targets $D$. \qed
\end{mythm}

In particular, since the property of being a targetable subspace of a Choquet space is upward hereditary, we obtain the following result.

\begin{mycor}	\label{thm:FT1BaireSpace}
Both $\SCinjectives$  and $\SCfins$, as subspaces of $\Cwstar$, are Choquet and hence Baire. \qed
\end{mycor}

Using the previous remark, NE also has a winning $1$-tactic in the Choquet game on $\SCinjectives$ (and on $\SCfins$), and both $\SCinjectives$ and $\SCfins$ are strongly Choquet.

\section{\texorpdfstring{$P$-points and weak $P$-points in $\CkX$}{P-points and weak P-points} }
\label{sectionPpoints}

We show that for an infinite compact zero-dimensional  space $X$, no autohomeomorphism of $X$ can be a $P$-point in $\CkX$ and that it is independent of ZFC whether $\Cwstar$ contains $P$-points or not. On the other hand, we show that every autohomeomorphism of $\wstar$ is a weak $P$-point in $\Cwstar$. More generally, autohomeomorphisms and open finite-to-one maps in $\CkX$ are always weak $P$-points for compact Hausdorff $F$-spaces which are nowhere c.c.c.

A $P$\emph{-point} is a point $p$ such that any countable intersection of neighbourhoods of $p$ contains a neighbourhood of $p$. In other words, $p$ is a $P$-point if $p$ is in the interior of every $G_\delta$-set containing $p$. The existence of $P$-points in $\wstar$ was first shown as a consequence of the Continuum Hypothesis (CH) by Rudin in \cite{Rudin}. The existence of $P$-points can also be shown under MA+$\neg$CH \cite[2.5.5]{Intro}. In general, however, Shelah proved it consistent that $P$-points in $\wstar$ do not exist \cite[2.7]{Intro}. 

A \emph{weak $P$-point} is a point $p$ which does not lie in the boundary of any countable set. The ZFC-existence of weak $P$-points in $\wstar$ was first shown by Kunen in \cite{Kunen}. Kunen's result was subsequently generalised to wider classes of compact $F$-spaces. First, van Mill proved in \cite{Mill79} the existence of weak $P$-points for every compact crowded $F$-space of weight $\cont$ in which each non-empty $G_\delta$ has non-empty interior. The weight restriction in van Mill's result was subsequently removed by Bell in \cite{Bell}. One year later, Dow and van Mill proved the following theorem.

\begin{mythm}[Dow and van Mill, \cite{DowMill80}]
\label{ExistenceWeakPPoints}
Every compact nowhere c.c.c.\ $F$-space contains a weak $P$-point.
\end{mythm}
In \cite{Dow82}, Dow complemented this result and proved that c.c.c.\ compact $F$-spaces of weight at least  $\cont^+$ contain weak $P$-points, and that it is consistent that every non-separable compact c.c.c.\ $F$-space contains a weak $P$-point.

\subsection*{\texorpdfstring{$P$-points in $\CkX$}{P-points}} We now ask under what conditions $\CkX$ contains $P$-points. 

\begin{mylem}[{\cite[XII.1.2]{Dugundji}.}]
Every space $X$ embeds into $\CkX$. \qed
\end{mylem}

Indeed, the map $x \mapsto f_x$ sending a point to the corresponding constant function $f_x$ is an embedding. The next result shows that $P$-points and weak $P$-points are preserved by this embedding.

%

\begin{myobs}
\label{Ppointsexist}
Let $X$ be a compact space and $p$ be a (weak) $P$-point of $X$. Then $f_p$ is a (weak) $P$-point in $\CkX$.
\end{myobs}

\begin{proof}
If $p$ is a $P$-point in $X$, one checks that whenever $f_p \in \bigcap_{n\in \w} [C_n,U_n]$ for $C_n \subseteq X$ compact and $U_n \subseteq X$ open, then $f_p \in [X,\interior{\bigcap U_n}] \subseteq \bigcap_{n\in \w} [C_n,U_n]$.

If $p$ is a weak $P$-point and $\set{f_l}:{l \in \w} \subseteq \CkX \setminus \singleton{f_p}$ a countable set, pick points $x_l \in \textnormal{ran}(f_l) \setminus \singleton{p}$ for each $l \in \w$. Then $[X,X \setminus \closure{\set{x_l}:{l \in \w}}]$ is a neighbourhood separating $f_p$ from $\set{f_l}:{l \in \w}$ as required.
\end{proof}

In particular, it follows for compact spaces $X$ that $f_p \in \CkX$ is a (weak) $P$-point if and only if $p \in X$ is a (weak) $P$-point.

The \emph{evaluation map} $\function{\eval}{\CkX \times X}{X}$ is given by $\pair{f, x} \mapsto f(x)$. For a point $x \in X$, the \emph{evaluation at $x$} is the map $\eval_x \colon \CkX \to X$ given by $f \mapsto f(x)$. The next lemmas show that evaluation maps are continuous open functions (with respect to the compact-open topology). 

\begin{mylem}[{\cite[3.4.3]{Eng}}]
\label{evaliscts}
For a locally compact Hausdorff space $X$, the evaluation mapping $\eval$ is continuous with respect to the compact-open topology. Consequently, the evaluation map at $x$, $\eval_x$, is continuous for every $x \in X$. \qed
\end{mylem}

\begin{mylem}
\label{evaluation}
Let $X$ be a zero-dimensional locally compact space. For every point $x \in X$, the evaluation map at $x$, $\eval_x$, is a continuous open map.
\end{mylem}

\begin{proof}
Continuity follows from the previous lemma. To show that $\eval_x$ is open, it is enough to consider the image of a basic open set $\bigcap_{k \leq n} [A_k,B_k]$ where all $A_k$ and $B_k$ are compact clopen subsets of $X$. Note that by Corollary~\ref{bigintersectionresult} we may assume all $A_k$ to be pairwise disjoint.
However, it is easy to verify that
$$ev_x(\bigcap_{k \leq n} [A_k,B_k]) = \begin{cases} B_k &\textnormal{ if } x \in A_k, \\ X & \textnormal{ otherwise,} \end{cases} $$
which is an open set as required.
\end{proof}

\begin{mylem}
Let $\sequence{U_x}{x \in X}$ be a sequence of open sets in $\CkX$. Let $S = \Union_{x \in X} U_x \times \singleton{x}$. Then $\eval\of{S}$ is open. In particular $\eval$ is an open map.
\end{mylem}

\begin{proof} We have
$$ \eval\of{S} = \eval\of{\Union_{x \in X} U_x \times \singleton{x}} = \Union_{x \in X} \eval\of{U_x \times \singleton{x}} = \Union_{x \in X} \eval_x\of{U_x},$$
and since each $\eval_x$ is an open map by Lemma~\ref{evaluation}, the set  $\eval\of{S}$ is open. 

To show $\eval$ is an open map, simply note that every basic open set $U \times V$ in $\CkX \times X$ can be written in the above form: setting $U_x = U$ for $x \in V$ and $U_x = \emptyset$ for $x \not \in V$ gives $S=\Union_{x \in X} U_x \times \singleton{x} = U \times V$. Therefore, $\eval(U \times V)$ is open by the above.
\end{proof}

With the help of the evaluation map we can now characterise $P$-points in $\CkX$ for compact zero-dimensional $X$.

\begin{mylem}
\label{projectPpoints}
Suppose that $X$ is a zero-dimensional compact space. If $f \in \CkX$ is a $P$-point then all $y \in \textnormal{ran}(f)$ are $P$-points in $X$.
\end{mylem}

\begin{proof}
It is straightforward to verify that the image of a $P$-point under a continuous open mapping is a $P$-point. Thus, if $y=f(x) \in \textnormal{ran}(f)$ then $y=ev_x(f)$ and the result follows from Lemma~\ref{evaluation}. 
\end{proof}

\begin{mythm}
\label{WhatarePfunctions}
Suppose that $X$ is a zero-dimensional compact Hausdorff space. Then $P$-points in $\CkX$ are precisely those functions that assume finitely many values, all of which are $P$-points in $X$.
\end{mythm}

\begin{proof}
Generalising Lemma~\ref{Ppointsexist} gives that every function whose range consists of finitely many $P$-points of $X$ is itself a $P$-point of $\CkX$.

Conversely, using Lemma~\ref{projectPpoints} it only remains to show that any $P$-point $f \in \CkX$ has finite range. But otherwise, $\textnormal{ran}(f)=f(X)$ is an infinite compact Hausdorff space and therefore contains non-$P$-points \cite[4.3]{Rudin}. This contradicts Lemma~\ref{projectPpoints}.
\end{proof}

\begin{mycor}
\label{AutosnoP}
For an infinite zero-dimensional compact space $X$, no autohomeomorphism of $X$ can be a $P$-point in $\CkX$. \qed
\end{mycor}

\begin{mythm}
\label{consindepententPpoints}
It is consistent with and independent of ZFC whether the space $\Cwstar$ contains $P$-points.
\end{mythm}

\begin{proof}
It follows from Theorem~\ref{WhatarePfunctions} that $\Cwstar$ contains $P$-points if and only if $\wstar$ contains $P$-points. The latter statement is well known to be consistent with and independent of ZFC (see \cite{Intro}).
 \end{proof}

\subsection*{\texorpdfstring{Weak $P$-points in $\CkX$}{Weak P-points}}
We now show that even though autohomeomorphisms are never $P$-points, they are weak $P$-points in $\Cwstar$. The result generalises to show that for compact zero-dimensional nowhere c.c.c.\ $F$-spaces, all autohomeomorphisms of $X$ are weak $P$-points in $\CkX$.

Note that the analogue of Theorem~\ref{projectPpoints} fails for weak $P$-points: indeed, weak $P$-points are preserved by a continuous open function only if that function has countable fibers, an assumption which does not hold for $\eval$.

\begin{mylem}
\label{Fspace}
Let $X$ be an $F$-space. If $A$ and $B$ are countable subsets of $X$ such that $\closure{A} \cap B = \emptyset = A \cap \closure{B}$, then $A$ and $B$ have disjoint closures. Moreover, if $X$ is compact zero-dimensional, there is a clopen subset $C \subset X$ such that $A \subset C \subset X \setminus B$.
\end{mylem}
\begin{proof}
The condition $\closure{A} \cap B = \emptyset = A \cap \closure{B}$ implies that $A$ and $B$ are contained in disjoint cozero sets. By the $F$-space property, $A$ and $B$  have disjoint closures. A straightforward application of zero-dimensionality and compactness yields the clopen set $C$.
\end{proof}

\begin{mylem}[Convergence Lemma]
\label{fundamental}
Let $X$ be a compact $F$-space. If for $f$ and $f_n \in \CkX$ ($n \in \w$) there are points $P = \set{p_n}:{n \in \w}$ such that for $A=f(P)$ and $B = \set{f_n(p_n)}:{n \in \w}$ we have $\closure{A} \cap B = \emptyset = A \cap \closure{B}$ then $f \notin \closure{\set{f_n}:{n \in \w}}$.
\end{mylem}

\begin{proof}
By the previous lemma, we have $\closure{A} \cap \closure{B} = \emptyset$. It follows that $[\closure{P},X \setminus \closure{B}]$ is a neighbourhood of $f$ disjoint from $\set{f_n}:{n \in \w}$.
\end{proof}

\begin{mythm}
\label{autoweakPpoints}
Let $X$ be a compact Hausdorff, nowhere c.c.c.\ $F$-space. Then every open finite-to-one map is a weak $P$-point in $\CkX$.
\end{mythm}

\begin{proof}
Let $X$ be a compact nowhere c.c.c.\ $F$-space. Since regular closed sets, i.e.\ sets of the form $\closure{U}$ for $U \subseteq X$ open, inherit these three properties, Theorem~\ref{ExistenceWeakPPoints} implies that such a space contains a dense set of weak $P$-points. Moreover, note that nowhere c.c.c.\ implies that countable subsets of $X$ are nowhere dense. 

Now let $f \in \CkX$ be an open map with finite fibres and let $\set{f_l}:{l \in \w} \subseteq \CkX \setminus \singleton{f}$ be an arbitrary countable collection of functions. We show that $f$ does not lie in the closure of $\set{f_l}:{l \in \w}$. 

For every $n$ let us consider $\Set{f_n\neq f} = \set{x \in X}:{f_n(x) \neq f(x)}$, a non-empty open subset of $X$. Let $F = \set{n \in \w}:{\cardinality{f_n(\Set{f_n\neq f} )} < \infty}$ and define 
$$E = \bigcup_{n \in F} f_n(\Set{f_n\neq f} ),$$ a countable, and hence nowhere dense subset of $X$.

By recursion on $l$, we pick weak $P$-points $\set{x_l}:{l \in \w}$ of $X$ (not necessarily faithfully indexed) such that 
\begin{enumerate}
\item $\set{f(x_l)}:{l \in \w} \cap \set{f_l(x_l)}:{l \in \w} = \emptyset$, and
\item $\set{f(x_l)}:{l \in \w} \cap E= \emptyset$.
\end{enumerate}
Moreover, whenever $f_m(x_m)$ lies in the boundary of a countable set of weak $P$-points of $X$, we will pick such a set $Y_m$, making sure that 
\begin{enumerate} \setcounter{enumi}{2}
\item $Y_m \cap \set{f(x_l)}:{l \in \w}  = \emptyset$ for all $m \in \w$. 
\end{enumerate}
This implies the result: letting 
$$A = \set{f(x_l)}:{l \in \w} \quad \textnormal{and} \quad B= \set{f_l(x_l)}:{l \in \w},$$
we see that since by assumption on $f$ all points in $A$ are weak $P$-points of $X$, requirement $(1)$ implies $A \cap \closure{ B}=  \emptyset$. Next, we claim (1) and (3) imply $\closure{A} \cap B= \emptyset$. To see this, suppose for a contradiction that $f_m(x_m) \in \closure{A}$. Clearly, then, the point $f_m(x_m)$ lies in the closure of a countable set of weak $P$-points, so condition (3) applies. But then the disjoint sets of weak $P$-points $Y_m$ and $A$ of $X$ do not have disjoint closures, contradicting Lemma~\ref{Fspace}. This establishes the claim. 

Thus,  we have
$$\closure{A} \cap B = \emptyset = A\cap \closure{ B},$$
and hence the result follows from the Convergence Lemma~\ref{fundamental}. 

It remains to describe the recursive construction. To begin, note that since $f$ has finite fibres, the set $f^{-1}(E)$ is countable and hence nowhere dense, so we may pick a weak $P$-point $x_0$ in $ \Set{f_0 \neq f} \setminus f^{-1}(E)$. Moreover, if possible we fix a countable set $Y_0$ of weak $P$-points different from $x_0$ such that $f_0(x_0)$ lies in the closure of $Y_0$. Otherwise, put $Y_0=\emptyset$.

For the inductive step, assume that the construction has been carried out successfully up to some $n \in w$, i.e.\ we have weak $P$-points $\set{x_l}:{l \leq n}$ and countable collections $Y_l$ of weak $P$-points for $l \leq n$ such that 
\begin{enumerate}
\item $\set{f(x_l)}:{l \leq n} \cap \set{f_l(x_l)}:{l \leq n} = \emptyset$, 
\item $\set{f(x_l)}:{l \leq n} \cap E= \emptyset$, and also
\item $Y_m \cap \set{f(x_l)}:{l \leq n} = \emptyset$ for all $m \leq n$. 
\end{enumerate}
We have to choose a point $x_{n+1}$ satisfying requirements (1) -- (3). Note that
$$\Set{f_{n+1} \neq f} \setminus f^{-1}_{n+1}(\set{f(x_l)}:{l \leq n}) \neq \emptyset.$$
This holds, as otherwise $f_{n+1}(\Set{f_{n+1} \neq f} )$ is contained in the finite set $\set{f(x_l)}:{l \leq n}$, yielding $\set{f(x_l)}:{l \leq n} \subset E$, and hence contradicting our induction assumption (2). 

Thus, the set $\Set{f_{n+1} \neq f} \setminus f^{-1}_{n+1}(\set{f(x_l)}:{l \leq n})$ is a non-empty open set and hence, since countable sets are nowhere dense, we can find a weak $P$-point
$$x_{n+1}  \in \Set{f_{n+1} \neq f} \setminus \p{f^{-1}_{n+1}(\set{f(x_l)}:{l \leq n}) \cup f^{-1}\p{\set{f_l(x_l)}:{l \leq n} \cup E} \cup \bigcup_{l \leq n} Y_l }.$$ This choice satisfies (1) and (2). Finally, if $f_{n+1}(x_{n+1})$ lies in the closure of a countable set of weak $P$-points, fix any such countable set $Y_{n+1}$ with $Y_{n+1} \cap \set{f(x_l)}:{l \leq n+1} = \emptyset$. Otherwise, put $Y_{n+1}=\emptyset$. This satisfies (3) and completes the recursion and proof.
\end{proof}

Note that in a compact $F$-space, infinite closed subsets contain a copy of $\beta \w$ and therefore have cardinality at least $2^\cont$. Thus, every self-map on such a space with fibres of size $< 2^\cont$ has finite fibres.

\begin{mycor}
\label{denseweakPpoints}
For a Parovi\v{c}enko space $X$ the function space $\CkX$ contains a dense set of weak $P$-points. In particular, the extensions of injective maps on $\w$ form a dense set of weak $P$-points in $\Cwstar$.
\end{mycor}
\begin{proof}
Combine~\ref{bigintersectionresult},~\ref{thm:FT1dense} and~\ref{autoweakPpoints}. Note that injective functions on $\w$ extend to injective functions on $\wstar$. 
\end{proof}

\section{\texorpdfstring{Convergent sequences in $\CkX$}{Convergent sequences}}

As a compact $F$-space, the space $\wstar$ does not contain non-trivial convergent sequences. Even though $\Cwstar$ is not an $F$-space (Theorem~\ref{thm:CkNotFspace}), we establish that $\Cwstar$ does not contain non-trivial convergent sequences (Theorem~\ref{noconvergentsequences}).

\begin{mylem}
\label{richardmaxcombined}
Let $X$ be a compact space and suppose $f_n \convergesTo f$ in $\CkX$. For any collection $P = \set{p_n}:{n \in \omega}$ of points in $X$, and $A = \image{P}{f} = \set{f\of{p_n}}:{n \in \omega}$ and $B = \set{f_n\of{p_n}}:{n \in \omega}$ we have
\begin{enumerate}
\item $\closure{A} = \closure{\image{P}{f}} = \image{\closure{P}}{f}$,
\item $\accumulation{A} \subseteq \closure{B}$, and
\item $\accumulation{B} \subseteq \closure{A}$.
\end{enumerate}
\end{mylem}

\begin{proof}
For $(1)$, note that $\image{\closure{P}}{f} \subseteq \closure{\image{P}{f}}$ by continuity of $f$, and we obtain equality because the image of $\closure{P}$ under $f$ must be compact and so closed.

For $(2)$, note that by continuity of $\eval$ (Lemma~\ref{evaliscts}), we have
\[ \image{\closure{\set{\pair{f_n, p_n}}:{n \in \omega}}}{\eval} \subseteq \closure{\image{\set{\pair{f_n, p_n}}:{n \in \omega}}{\eval}} = \closure{B}. \]

Now if $x \in \accumulation{A}$, then by (1) we can write $x = f(p)$ for $p \in \accumulation{P}$. We show $\pair{f, p} \in \closure{\set{\pair{f_n, p_n}}:{n \in \omega}}$, and hence, by the above, $x \in \closure{B}$. So let $U$ be an open neighbourhood of $f$ in $\CkX$, and let $V$ be an open neighbourhood of $p$ in $X$. Since $f_n \convergesTo f$, there is an $N \in \omega$ such that $\set{f_n}:{n \geq N} \subseteq U$. Since $p \in \closure{P}$, we have $P \intersection V$ is infinite, and hence must contain some $p_n$ for $n \geq N$. Then $\pair{f_n, p_n} \in U \times V$. This proves $\pair{f, p} \in \closure{\set{\pair{f_n, p_n}}:{n \in \omega}}$, and hence (2).

For $(3)$, suppose for a contradiction that there is $x \in \closure{B} \setminus \p{B \cup \closure{A}}$. Pick an open neighbourhood $U$ of $x$ such that $\closure{U} \cap \closure{A} = \emptyset$. Since $x \in \accumulation{B}$, the set $I = \set{n \in \w}:{f_n\of{p_n} \in U}$ must be infinite. Using $(1)$, we conclude that $f \in [\closure{P}, \wstar \setminus \closure{U}]$ but $f_n \notin [\closure{P}, \wstar \setminus \closure{U}]$ for all $n \in I$. This contradicts  $f_n \convergesTo f$.
\end{proof}

\begin{mylem}
\label{witnessingP}
Let $X$ be a compact space without non-trivial convergent sequences and let $f$ and $f_n$ (for $n \in \w$) be pairwise distinct functions in $\CkX$. Then there exists a subsequence $\set{f_{n_i}}:{i \in \w}$ and points $P = \set{p_i}:{i \in \w}$ such that $A = f(P)$ and $B = \set{f_{n_i}(p_i)}:{i \in \w}$ are disjoint.
\end{mylem}

\begin{proof}
We will differentiate between three cases.

\begin{mycase}
For some $F = f^{-1}(x)$, the functions $f_n\restriction{F}$ do not equal $f\restriction{F}$ eventually.
\end{mycase}

In this case, there is an infinite subsequence $\set{f_{n_i}}:{i \in \w}$ and $p_i \in F$ such that $f_{n_i}(p_i) \neq x$ and we are done.

Next, for $n \in \w$ let us consider $\Set{f_n\neq f} = \set{x \in X}:{f_n(x) \neq f(x)}$, a non-empty open set. Note that since we have dealt with Case 1, we may from now on assume that for all \emph{finite} $A \subset X$ we have $f^{-1}(A) \cap \Set{f_n\neq f} = \emptyset$ eventually $(\star)$. For the remaining two cases, consider the sets $E_k = \bigcup_{n \geq k} f_n(\Set{f_n\neq f} )$.

\begin{mycase}
Property $(\star)$ holds and the sets $E_k$ are eventually finite.
\end{mycase}
Let $E$ denote the first $E_k$ which is finite. By $(\star)$ there exists $N \geq k$ such that 
$$f^{-1}(E) \cap \Set{f_{n} \neq f} = \emptyset$$
 for all $n \geq N$, so pick points $p_n \in \Set{f_{n} \neq f}$ for each $n \geq N$. It follows that 
$$\set{f\of{p_n}}:{n \geq N} \cap \set{f_n\of{p_n}}:{n \geq N} \subseteq \p{X \setminus E} \cap E = \emptyset,$$
as desired. 

\begin{mycase}
Property $(\star)$ holds and all $E_k$ are infinite.
\end{mycase}

We use recursion to find an infinite subsequence $S=\set{n_l}:{l \in \w}$ and points $\set{p_l}:{l \in \w}$ such that $\set{f\of{p_l}}:{l \in \omega} \cap \set{f_{n_l}\of{p_l}}:{l \in \omega}=\emptyset$. To begin, pick any $p_0 \in \Set{f_0 \neq f}$. For the recursion step, assume we have found $\set{p_l}:{l \leq k}$ such that 
$$A = \set{f(p_l)}:{l \leq k} \quad \text{and} \quad B= \set{f_{n_l}(p_l)}:{l \leq k}$$
are disjoint. Again, by $(\star)$, there is $N \geq n_k$ such that 
$$f^{-1}(B) \cap \Set{f_n \neq f} = \emptyset \quad \text{for all } n \geq N.$$ Next, we claim that there is some $n > N$ such that 
$$\Set{f_n \neq f} \setminus f_n^{-1}(A) \neq \emptyset.$$
Indeed, otherwise we would have $f_n(\Set{f_n \neq f}) \subseteq A$ for all $n > N$, i.e. $E_k \subset A$ eventually, a contradiction. Let $n_{k+1} > N$ be such an $n$.

Now pick $p_{k+1} \in \Set{f_{n_{k+1}} \neq f} \setminus f^{-1}_{n_{k+1}}(A)$. We have $f(p_{k+1}) \notin B$, and therefore it follows that 
$$\set{f\of{p_l}}:{l \leq k+1} \cap \set{f_{n_l}\of{p_l}}:{l \leq k+1}=\emptyset$$
as desired. This completes the recursion. 
\end{proof}

\begin{mythm}
\label{noconvergentsequences}
The space of self-maps $\CkX$ of a compact $F$-space $X$ does not contain non-trivial convergent sequences.
\end{mythm}

\begin{proof}
Suppose for a contradiction that $f_n \convergesTo f \in \CkX$. Moving to a subsequence if necessary, we may assume by Lemma~\ref{witnessingP} that there are points $P= \set{p_n}:{n \in \w}$ such that $A=f(P)$ and $B = \set{f_n(p_n)}:{n \in \w}$ are disjoint. Note that since in a compact $F$-space, infinite closed subsets have size at least $2^\cont$, it follows from Lemma~\ref{richardmaxcombined} that either $A$ and $B$ are both finite or both infinite. If they are finite, then $\closure{A} \cap \closure{B} = \emptyset$ and we are done by the Convergence Lemma~\ref{fundamental}.

So assume that both $A$ and $B$ are infinite. Since every infinite regular space contains an infinite discrete subspace \cite[VII.2.4]{Dugundji}, we may assume, after moving to another subsequence of the $f_n$, that $A$ is infinite discrete. As before, by Lemma~\ref{richardmaxcombined} it follows that the corresponding $B$ is still infinite, so after moving to another subsequence, we can assume that both $A$ and $B$ are infinite and discrete. Since countable subsets of $F$-spaces are $\Cstar$-embedded, it follows that $\closure{A} \setminus A \cong \wstar \cong \closure{B} \setminus B$.

Now consider $A \setminus \closure{B}$. If this set is finite, then, after moving to a tail of our sequence, we may assume that $A \setminus \closure{B} = \emptyset$, and hence that $A \subseteq \closure{B} \setminus B \cong \wstar$. But since $A$ is countable and $\wstar$ has density $\cont$, it follows that $\closure{A} \subsetneq \closure{B} \setminus B$, contradicting Lemma~\ref{richardmaxcombined}(3). Thus, we may assume that the set $A \setminus \closure{B}$ is infinite, and hence, after moving to another subsequence, that $ A \cap \closure{B} = \emptyset.$

Next, consider $B \setminus \closure{A}$. Similarly to the previous case, the finiteness of this set will contradict Lemma~\ref{richardmaxcombined}(2). Thus, we may assume that the set $B \setminus \closure{A}$ is infinite, and hence, after moving to another subsequence, that $ B \cap \closure{A} = \emptyset.$

Thus, we have found a subsequence, say $\set{f_{n_i}}:{i \in \w}$, such that for the sets $A=\set{f(p_{n_i})}:{i \in \w}$ and $B = \set{f_{n_i}(p_{n_i})}:{i \in \w}$ we have 
$$ A \cap \closure{B} = \emptyset = B \cap \closure{A}.$$ It follows from the Convergence Lemma~\ref{fundamental} that $f$ is not an accumulation point of  $\set{f_{n_i}}:{i \in \w}$, contradicting $f_n \convergesTo f$. This last contradiction proves the theorem.
\end{proof}

As our final result we show that $\CkX$ is not an $F$-space, for any infinite zero-dimensional compact Hausdorff space $X$. In what follows, we let $\set{A_n}:{n \in \w}$ be an infinite collection of disjoint clopen subsets of $X$. We construct two open disjoint $F_\sigma$ sets which contain the identity, $\id$, in their closure. 

\begin{mylem}
\label{nhoodsidentity}
Let $X$ be a locally compact zero-dimensional space. Then the identity $\id \in \CkX$ has a neighbourhood basis of sets of the form $[A_1,A_1] \cap \cdots \cap [A_N,A_N]$ for $A_n \subset X$ disjoint clopen subsets of $X$. \qed
\end{mylem}

\begin{mylem}
\label{intersections}
For $A,B$ disjoint clopen subsets of $X$, and $C$ any third clopen set of $X$ we have $[C,C] \cap [A,B] = \emptyset$ if and only if both $C \cap A \neq \emptyset$ and $C \cap B = \emptyset$. 
\end{mylem}

\begin{proof}
For the forwards direction, note that if the second condition is violated then $[C,C] \cap [A,B] \supseteq [X,C \cap B] \neq \emptyset$. And if the first condition is violated then $[C,C] \cap [A,B] \supseteq [C,C] \cap [X\setminus C,B] \neq \emptyset$.
  
Conversely, consider $f \in [C,C]$. If $C \cap A \neq \emptyset$ and $C \cap B = \emptyset$ then for any $x \in C \cap A$ we have $f (x) \in C$, so $f(x) \notin B$ and hence $f \notin [A,B]$.
\end{proof}

\begin{mylem}	\label{lem:approachi}
Let $V=[C_1,C_1] \cap \cdots \cap [C_N,C_N]$ be a basic open neighbourhood of the identity $\id \in \CkX$. Then if $I \subseteq \omega$ with $\cardinality{I} \geq 2^{N}$, then writing
\[ U_I = \bigcup_{n \neq m \in I} [A_n,A_m], \]
we have
\[ V \intersection U_I \neq \emptyset. \]
\end{mylem}

\begin{proof}
For $N=1$, note that Lemma~\ref{intersections} says that $[C_1,C_1]$ intersects any $[A_n,A_m] \cup [A_m,A_n]$ for $n \neq m$. Using induction on $N$, we now we prove that $V$ intersects $U_I$. Let us first consider $[C_N,C_N] \cap U_I$. Recall that by Lemma~\ref{intersections}, $[C_N,C_N]$ does not intersect $[A_n,A_m]$ if and only if $C_N \cap A_m = \emptyset$ and $C_N \cap A_n \neq \emptyset$. Now put $I_N = \set{n \in I}:{A_n \cap C_N = \emptyset}$. Then $[C_N,C_N]$ in particular intersects every single set in $\set{[A_n,A_m]}:{n \neq m \in I_N}$ and $\set{[A_n,A_m]}:{n \neq m \in I \setminus I_N}$.

In other words, for $V$ to be disjoint from $U_I$ we must have $V'=[C_1,C_1] \cap \cdots \cap [C_{N-1},C_{N-1}]$ disjoint from $\bigcup_{n \neq m \in I_N} [A_n,A_m] \cup \bigcup_{n \neq m \in I \setminus I_N} [A_n,A_m]$. However, either $\cardinality{I_N}$ or $\cardinality{I \setminus I_N}$ is bigger than $2^{N-1}$, and therefore this is impossible by our induction assumption.
\end{proof}

Note that it follows that the identity is in the boundary of the open $F_\sigma$-set $U = \bigcup_{n \neq m \in \omega} [A_n,A_m]$. We now divide this open $F_\sigma$ into two halves, each containing the identity in its boundary. For $n, m \in \omega$ define
\[ \operator{Fix}\of{n,m} = \Intersection_{\substack{k < \operator{max}\of{n,m} \\ k \neq n,m}} \CompactOpen{A_k}{A_k}, \]
noting that $\operator{Fix}\of{n,m}$ is clopen in $\CkX$, and $\operator{Fix}\of{n,m} = \operator{Fix}\of{m,n}$.

\begin{mylem}	\label{lem:refinedApproachToi}
Suppose $V=[C_1,C_1] \cap \cdots \cap [C_N,C_N]$ is a basic open neighbourhood of the identity, and $V$ intersects $U = \bigcup_{n \neq m \in I} [A_n,A_m]$ for some $I \subseteq \omega$. Then $V$ also intersects
\[ \bigcup_{n \neq m \in I} \left ( [A_n,A_m] \intersection \operator{Fix}\of{n,m} \right )\]
\end{mylem}

\begin{proof}
Fix $f \in V \intersection U$. Since $f \in U$, there are indices $k\neq l \in I$ such that $f \in \CompactOpen{A_k}{A_l}$. Now define 
$$g = f\restriction{A_k} \cup \; \id \restriction{X \setminus A_k} \in \CkX.$$
One checks that $g \in V \cap \bigcup_{n \neq m \in I} \left ( [A_n,A_m] \intersection \operator{Fix}\of{n,m} \right )$.
\end{proof}

\begin{mythm}	\label{thm:CkNotFspace}
For an infinite zero-dimensional compact Hausdorff space $X$, its function space $\CkX$ is not an $F$-space.
\end{mythm}

\begin{proof}
As above, let $\set{A_n}:{n \in \w}$ be an infinite collection of disjoint clopen subsets of $X$. Now define
\[ U_E = \Union_{n \neq m \in \omega} [A_{2n}, A_{2m}] \intersection \operator{Fix}\of{2n,2m} \]
and
\[ U_O = \Union_{n \neq m \in \omega} [A_{2n+1}, A_{2m+1}] \intersection \operator{Fix}\of{2n+1,2m+1}. \]
Then $U_E, U_O$ are open $F_\sigma$ subsets of $\CkX$. The reader is invited to verify the following two claims.

\begin{myclaim}
$U_E \intersection U_O = \emptyset$.
\end{myclaim}

%

\begin{myclaim}
We have $i \in \closure{U_E} \setminus U_E$. Similarly $i \in \closure{U_O} \setminus U_O$.
\end{myclaim}

The proof of the first claim is left as an exercise; the second claim follows from Lemmas~\ref{lem:approachi} and \ref{lem:refinedApproachToi}.
Thus, we have found two disjoint cozero-sets of $\CkX$ with intersecting closures. Hence, $\CkX$ is not an $F$-space.
\end{proof}

Lastly, we remark that $\Cwstar$ also fails to have the $G_\delta$ property: $\Cwstar$ contains a non-empty $G_\delta$ with empty interior. Indeed, let $A\oplus B$ be a non-trivial clopen partition of $\wstar$, and let $f$ be an autohomeomorphism of $\wstar$ swapping $A$ and $B$. If we let $\set{A_n}:{n\in \w}$ be a collection of disjoint clopen sets contained in $A$, then $ \bigcap_{n \in \w} [A_n,f(A_n)]$ is a non-empty $G_\delta$ with empty interior.


\section{Open Questions}

We have seen in Section~\ref{sectionPpoints}  that $\Cwstar$ is not homogeneous: it contains weak $P$-points (for example, all autohomeomorphisms of $\wstar$ and constant functions $f_p$ with $p \in \wstar$ a weak $P$-point) and it contains non-weak $P$-points (for example, constant functions $f_x$ where $x \in \wstar$ is not a weak $P$-point). We also saw that consistently, $\Cwstar$ also contains $P$-points. Thus, we have found two, and consistently three, orbits of $\Cwstar$.

\begin{myquest}
What is the number of orbits of $\Cwstar$?
\end{myquest}

\begin{myquest}
Is it true that if $x,y \in \wstar$ are in different orbits of $\wstar$ then $f_x,f_y$ are in different orbits of $\Cwstar$? 
\end{myquest}

By a result of Frol\'ik \cite{Frolik}, a positive answer would imply that $\Cwstar$ has the maximal number of distinct orbits, namely $2^\cont$.

\begin{myquest}
Can a constant weak $P$-point function $f_p$ and an autohomeomorphisms $f \in \Cwstar$ lie in the same orbit?
\end{myquest}

The last question is interesting in light of the fact that cardinality of the range of a function is not an invariant of orbits of $\Cwstar$: the following observation, which was pointed out to us by R.\ Suabedissen, shows that constant maps can lie in the same orbit as some maps with finite range.

\begin{myobs}
Let X be a locally compact zero-dimensional  space. For any clopen partition $X = A_1 \oplus \ldots \oplus A_n$ and autohomeomorphisms $h_i \colon X \rightarrow X$ for $i \leq n$,  $$\Psi \colon \CkX \rightarrow \CkX, f \mapsto f \circ \p{h_1|_{A_1} \cup \cdots \cup h_n|_{A_n}}$$ is an autohomeomorphism of $\CkX$.
\end{myobs}

\begin{proof}
  To see that $\Psi$ is continuous, note that by zero-dimensionality, $h_1|_{A_1} \cup \cdots \cup h_n|_{A_n}$ is a continuous map $X \to X$. We know that the compact-open topology is proper and admissable, and \cite[Exercise 2.6.D(c)]{Eng} says that for proper and admissable topologies on spaces of self-maps, composition is a continuous map. Hence $\function{\circ}{\CkX \times \CkX}{\CkX}$ is continuous. One then easily sees that $\Psi$ is continuous.

The inverse $\Psi^{-1} \colon \CkX \rightarrow \CkX, f \mapsto f\circ \p{h^{-1}_1|_{A_1} \cup \cdots \cup h^{-1}_n|_{A_n}}$ is continuous by the same argument as above.
\end{proof}

\begin{myquest}
Is the space $\Cwstar$ normal? Is it \v{C}ech-complete?
\end{myquest}

\begin{myquest}
For which compact zero-dimensional Hausdorff spaces $X$ is the function space $\CkX$ Baire? What about $X = 2^{\w_1}$?
\end{myquest}

	  
\end{document}